\documentclass[letterpaper,10pt,pdflatex,reqno]{amsart}
\usepackage{amsmath,times} \usepackage{amsthm}
\usepackage[textsize=tiny]{todonotes}
\usepackage{graphicx}
\usepackage{color}

\usepackage{amssymb} \usepackage{latexsym} \usepackage{amscd}
\usepackage{hyperref}
\usepackage{verbatim} 
\newtheorem{theorem}{Theorem}[section]
\newtheorem{theorem*}{Theorem}
\newtheorem{lemma}[theorem]{Lemma}
\newtheorem{proposition}[theorem]{Proposition}

\theoremstyle{definition} 
\newtheorem{definition}[theorem]{Definition}
\newtheorem*{definition*}{Definition}

\newtheorem*{remark*}{Remark}

\newtheorem{question}[theorem]{Question}

\begin{document}

\begin{abstract}
We give examples of regular boundary data for the Dirichlet problem for the Complex Homogeneous Monge-Amp\`ere Equation over the unit disc, whose solution is completely degenerate on a non-empty open set and thus fails to have maximal rank.
\end{abstract}

\title[The Maximal Rank Problem for the HMAE]{On the Maximal Rank Problem for the \\Complex Homogeneous Monge-Amp\`ere Equation}

\author{Julius  Ross and David Witt Nystr\"om}
\maketitle
\section{Introduction}
Let $(X,\omega)$ be a compact K\"ahler manifold of dimension $n$ and $B$ be a Riemann surface with boundary $\partial B$.  Suppose $(\phi_\tau)_{\tau\in \partial B}$ is a smooth family of K\"ahler potentials on $X$;  so each $\phi_\tau$ is a smooth function on $X$, varying smoothly in $\tau$, that satisfies
$$ \omega + dd^c\phi_\tau >0.$$
Then let $\Phi$ be the solution to the Dirichlet problem for the complex Homogeneous Monge-Amp\`ere Equation (HMAE) with this boundary data, so $\Phi$ is a function on $X\times B$ that satisfies
\begin{align}\label{eq:hmae}
\Phi(\cdot,\tau) = \phi_\tau(\cdot) &\text{ for } \tau\in \partial B\\
\pi_{X}^*\omega + dd^c\Phi&\ge 0, \nonumber \\
(\pi_{X}^* \omega + dd^c\Phi)^{n+1}&=0, \nonumber
\end{align}
where $\pi_X:X\times B\to X$ is the projection.  From standard pluripotential theory we know there exists a unique weak solution $\Phi$ to this equation.  The \emph{maximal rank problem} in this setting asks whether the current 
$$\pi_X^*\omega + dd^c\Phi$$
has maximal rank in the fibre directions, that is whether the current $\omega + dd^c\Phi(\cdot,\tau)$ on $X$ is strictly positive for each $\tau\in B$.  Said another way this asks if the rank of $\pi_X^*\omega + dd^c\Phi$ is precisely $n$ at every point in $X\times B$, which is the maximum possible since $(\pi_X^*\omega + dd^c\Phi)^{n+1}=0$. Similarly one has the \emph{constant rank problem} in which one asks if the rank of $\pi_X^*\omega + dd^c\Phi$ is the same at every point.    The purpose of this note is to answer this question negatively, giving an explicit example in which the rank  fails to be maximal.

\begin{theorem}\label{thm:mainimprecise}
Let $B=\overline{\mathbb D}\subset \mathbb C$ be the closed unit disc and $(X,\omega)=(\mathbb P^1,\omega_{FS})$ where $\omega_{FS}$ denotes the Fubini-Study form.  Then there exists a smooth family of K\"ahler potentials $(\phi_\tau)_{\tau\in \partial \mathbb D}$ on $\mathbb P^1$ such that the solution $\Phi$ to the HMAE \eqref{eq:hmae} is completely degenerate on some non-empty open subset $S\subset \mathbb P^1\times \mathbb D$, i.e.
$$ \pi_{\mathbb P^1}^*\omega_{FS} + dd^c\Phi |_S =0.$$
\end{theorem}

A more precise version of this statement is provided in Theorem \ref{thm:mainprecise}.  The motivation and ideas build on previous work of the authors \cite{RWDisc, RW, RWApplications} in which we understand the solution to the HMAE of a certain kind through a free boundary problem in the plane called the Hele-Shaw flow.  But rather than expecting the reader to be an expert in this topic we have chosen to give a direct proof, which can be found in Section \ref{sec:maintheorem}, that is both self-contained and rather simple.  Then in Section \ref{sec:discussion} we explain the motivation behind our construction, as well as give a second (but essentially equivalent) proof that relies on more machinery.  We then end with some questions and possible extensions.  

Of course in the above Theorem, $\pi_{\mathbb P^1}^* \omega_{FS} + dd^c\Phi$ is not identically zero, and so does not have constant rank.  In fact we can say more, and it is possible to arrange so that there is a non-empty open set in $\mathbb P^1\times \mathbb D$ on which $\pi_{\mathbb P^1}^* \omega_{FS} + dd^c\Phi$ is regular (i.e. smooth and of maximal rank).    It is worth commenting from the outset that we do not expect the solution we have here to be everywhere smooth, but it should be possible to describe precisely where it is regular and where it is degenerate.   All of this will be discussed in more detail in Section \ref{sec:discussion}.

\subsection{Comparison with other work}
It is known that convex solutions to elliptic partial differential equations have a constant rank property.  Early works of this include that of Caffarelli-Friedman \cite{CafarelliFriedman} and Yau \cite{SingerYau}.  These have since been built upon by many others, and it is now known that the constant rank property holds for a wide class of elliptic equations (see, for instance, Korevaar-Lewis \cite{KorevaarLewis}, Bian-Guan \cite{BianGuan1,BianGuan2}, Caffarelli-Guan-Ma \cite{CafarelliGuanMa}, Sz\'ekelyhidi-Weinkove \cite{SzWeinkove}).    In this paper we are interested in the complex degenerate  situation, about which much less has been written.  The most famous result along these lines, and in the positive direction, is that of Lempert \cite{Lempert} who proves that on a convex domain in $\mathbb C^n$ the solution to the complex HMAE with prescribed singularity at an interior point (the pluri-complex Green function)  is smooth and of maximal rank.   The maximal rank problem for other partial differential equations in the complex case has also been taken up by Guan-Li-Zhang \cite{GuanLiZhang} and by Li \cite{Qi}.

The closest previous work to that of this paper is probably that of Guan-Phong \cite{GuanPhong1}, who study the problem of finding uniform lower bounds for the eigenvalues of the  solution to the (non-degenerate) Monge-Amp\`ere Equation in the limit as the equation becomes degenerate.  Moreover,  they ask whether solutions to the complex  HMAE has maximal rank  \cite[discussion after Theorem 4]{GuanPhong1}.   The idea of maximal rank for the complex HMAE also appears in the ideas of Chen-Tian through the concept of an almost-regular solution to the HMAE, that fails to have maximal rank only on a set which is small in a precise sense (see \cite{ChenTian}). The kinds envelopes that we use in our proof also can be defined more generally, and even in higher dimensions, which is the topic of previous work of the authors \cite{RWTubular}, in which we prove a constant-rank theorem \cite[Theorem 6.2]{RWTubular} that we call ``optimal regularity".

Questions concerning the regularity of the solution to the Dirichlet problem for the kind of complex HMAE we consider here go back at least as far as the work of Semmes \cite{Semmes} and Donaldson \cite{Donaldson}, and has been the focus of much interest due to it being the geodesic equation in the space of K\"ahler metrics.   By the work of Chen \cite{Chen} with complements by B\l ocki \cite{Blocki} we know such a solution always has bounded Laplacian (so in particular is $C^{1,\alpha}$ for any $\alpha<1$).   In fact in our case, since we are working on $\mathbb P^1$, the results of \cite{Blocki} imply that $\Phi$ is $C^{1,1}$.  (We observe that we do not actually need to know this regularity for the direct proof of our main theorem). Donaldson gives in \cite{Donaldson} examples of boundary data for which the solution is not-regular, but the nature of the irregularity there is left unknown (for instance Donaldson's example may have maximal rank but fail to be everywhere smooth).\medskip

\noindent {\bf Acknowledgements: } The authors would like to thank Valentino Tossatti for discussions, in particular for pointing out that this problem was a topic for discussion at the American Institute of Mathematics (AIM) workshop on The Complex Monge-Amp\`ere Equation \cite{AIMMeeting}.   During this work JR was supported by an EPSRC Career Acceleration Fellowship (EP/J002062/1). 

\section{Main Theorem}\label{sec:maintheorem}

\subsection{Notations}
We let $\mathbb D_r$ be the open disc of radius $r$ in the complex plane about the origin, $\mathbb D=\mathbb D_1$ and $\mathbb D^\times = \mathbb D\setminus \{0\}$.  Throughout we consider the standard cover of $\mathbb P^1$ by two charts equal to the complex plane with coordinates $z$ and $w=1/z$.    We shall denote these two charts by $\mathbb C_z$ and $\mathbb C_w$ respectively.   We use the convention $d^c = \frac{1}{2\pi}(\overline{\partial} - \partial)$ so $dd^c \log |z|^2 = \delta_0$, and normalise the Fubini-Study form $\omega_{FS}$ so $\int_{\mathbb P^1} \omega_{FS}=1$.         Thus $\omega_{FS}= dd^c \log (1+|z|^2)$ locally on $\mathbb C_z$.

\subsection{Statement of Main Theorem}

The following is a precise version of our main theorem.   By an \emph{arc} in $\mathbb C$ we mean the image $\gamma$ of a smooth map $[0,1]\to \mathbb C$ that does not intersect itself.  From now on $B= \overline{\mathbb D}$ is the closed unit disc and $(X,\omega) = (\mathbb P^1,\omega_{FS})$. 

\begin{theorem}\label{thm:mainprecise}
Suppose that $\phi\in C^{\infty}(\mathbb P^1)$ satisfies
\begin{enumerate}
\item $\omega_{FS} + dd^c \phi>0$
\item On $\mathbb C_w\subset \mathbb P^1$ it holds that
$$\phi\ge -\ln (1+|w|^2)$$ 
with equality precisely on an arc in $\mathbb C_w$.
\end{enumerate}
Then setting 
$$\phi_\tau(z):=\phi(\tau z) \text{ for }\tau\in \partial \mathbb D,$$ the solution $\Phi$ to the HMAE \eqref{eq:hmae} does not have maximal rank.  In fact there is a non-empty open subset $S\subset \mathbb P^1\times \mathbb D$ such that
$$ \pi_{\mathbb P^1}^* \omega_{FS} + dd^c\Phi |_S =0.$$
\end{theorem}

 \subsection{Envelopes} 
For the proof we need some background concerning envelopes of subharmonic functions.   Fix a potential $\phi\in C^{\infty}(\mathbb P^1)$ so $\omega_{FS} + dd^c\phi>0$.

\begin{definition}\label{def:defofpsi}
For $t\in (0,1]$ set
$$\psi_{t} := \sup\{\psi\colon \mathbb P^1\to \mathbb R\cup \{-\infty\}: \psi \text{ is usc, }\psi\le \phi \text{ and } \omega_{FS} + dd^c\psi\geq 0 \text{ and } \nu_{z=0}(\psi)\ge t\}.$$
\end{definition}

Here $\nu_{z=0}$ denotes the Lelong number at the point $z=0$, so $\nu_{z=0}(\psi)\ge t$ means that $\psi(z)\le t\ln |z|^2 + O(1)$ near $z=0$.  As the upper-semi continuous regularisation of $\psi_t$ is itself a candidate for the envelope defining $\psi_t$, we see that $\psi_t$ is itself upper-semicontinuous. 

\begin{definition}
For $t\in (0,1]$ set
\begin{equation}
\Omega_t:=\Omega_t(\phi) : = \{ z\in \mathbb P^1 : \psi_t(z)<\phi(z)\}.\label{eq:defHS}
\end{equation} 
\end{definition}

%

Clearly if $t\le t'$ then $\psi_{t'}\le\psi_t$ and so $\Omega_t\subset \Omega_{t'}$.   Now, unless one assumes some additional symmetry of $\phi$, it is generally quite hard to describe the sets $\Omega_t$.  However as the next Lemma shows it is possible, under suitable hypothesis, to describe the largest one $\Omega_1$ by looking at the level set on which $\phi$ takes its minimum value.

\begin{lemma}\label{lem:omega1}
Let $\phi\in C^{\infty}(\mathbb P^1)$ be such that $\omega_{FS} + dd^c\phi>0$ and $\phi(w) \ge - \ln (1+|w|^2)$ on $\mathbb C_{w}$ with equality precisely on some non-empty subset $\gamma\subset \mathbb C_w$ containing $w=0$.   Then 
$$\psi_1(z) =\ln\left(\frac{|z|^2}{1+|z|^2}\right)$$
and
$$\Omega_1(\phi) = \mathbb P^1\setminus \gamma.$$
\end{lemma}
\begin{proof}
Observe first that the the only upper-semicontinuous $\psi:\mathbb P^1\to \mathbb R\cup \{-\infty\}$ with $\omega_{FS} +dd^c\psi\ge 0$ and $\nu_{z=0}\psi \ge 1$ is, up to an additive constant, equal to
$$ \zeta(z) := \ln\left(\frac{|z|^2}{1+|z|^2}\right)\text{ on } \mathbb C_z.$$
To see this observe first that we certainly cannot have $\nu_{z=0}(\psi)>1$ since we have normalised so $\int_{\mathbb P^1}\omega_{FS}=1$.  Thus we may assume $\nu_{z=0}\psi =1$.  Then observe that $\zeta$ is $\omega_{FS}$-harmonic on $\mathbb C_{z}\setminus \{0\}$, and that the difference $\psi-\zeta$ is bounded near $0$.  Thus  $\psi-\zeta$ extends to a bounded subharmonic function on all of $\mathbb C_z$, and hence is constant by the Liouville property.    Thus the envelope $\psi_1$ from Definition \ref{def:defofpsi} must be
$$\psi_1 = \zeta + C$$
where $C$ is the largest constant one can choose so that $\psi_1\le \phi$.   Now on $\mathbb C_w$ we have
$$ \psi_1(w) = -\ln (1+|w|^2) + C$$
and so as $\gamma$ is non-trivial our hypthesis force $C=0$. Thus 
$$\Omega_1  = \{ -\ln (1+|w|^2)< \phi(w)\}= \mathbb P^1\setminus \gamma.$$\end{proof}

\subsection{Weak solutions to the HMAE}

We now discuss the weak solution to two versions of the Dirichlet Problem for the complex HMAE, first over the disc and second over the punctured disc (this follows the discussion in \cite{RW}).    Again we let $\phi\in C^{\infty}(\mathbb P^1)$ be such that $\omega_{FS}+ dd^c\phi>0$.

\begin{definition}\

\begin{enumerate}
\item Let
\begin{equation*}\label{eq:weaksolutionD}
 \Phi := \sup 
 \left\{
\begin{array}{c}
  \psi \colon  \mathbb P^1\times \overline{\mathbb D}\to \mathbb R\cup \{-\infty\} : \psi \text{ is usc, }\pi_{\mathbb P^1}^*\omega_{FS} + dd^c\psi \ge 0 \\\text{ and } \psi(z,\tau)\le \phi(\tau z) \text{ for } (z,\tau)\in \mathbb P^1\times  \partial\mathbb D
 \end{array}
 \right\}.
\end{equation*}
\item Let
\begin{equation}\label{eq:weaksolutionDtimes}
\tilde{\Phi} :=\sup \left\{
\begin{array}{c}
\psi \colon \mathbb P^1\times \overline{\mathbb D}\to \mathbb R\cup \{-\infty\} : \psi\text{ is usc, } \pi_{\mathbb P^1}^*\omega_{FS} + dd^c\psi\ge 0 \\\text{and } \psi(z,\tau)\le \phi(z) \text{ for } (z,\tau)\in \mathbb P^1\times \partial \mathbb D \text{ and } \nu_{(z=0,\tau=0)}(\psi)\ge 1
\end{array}
\right\}.
\end{equation}
\end{enumerate}
\end{definition}

The function $\Phi$ is the weak solution to the complex HMAE with boundary data $\phi(\tau z)$ for $\tau \in \partial \mathbb D$, that is the solution to equation \eqref{eq:hmae}.    Similarly $\tilde{\Phi}$ is the weak solution to the Dirichlet problem with boundary data $\phi(z)$, but with the additional requirement of having a prescribed singularity at the point $p:= (0,0)\subset \mathbb C_z\times \mathbb D\subset \mathbb P^1\times \mathbb D$.  That is, $\tilde{\Phi}$ is usc,  $\pi_{\mathbb P^1}^*\omega_{FS} + dd^c\tilde{\Phi}\ge 0$ and $(\pi_{\mathbb P^1}^*\omega_{FS} + dd^c\tilde{\Phi})^2=0$ away from $p$ and $\tilde{\Phi}(z,\tau) = \phi(z)$ for $\tau \in \partial \mathbb D$.   Moreover it is not hard to show that $\tilde{\Phi}$ is locally bounded away from $p$ and $\nu_{p}\tilde{\Phi}=1$.       These two quantities carry the same information as given by:

\begin{proposition} \label{prop:equiv}
We have that $$\Phi(z,\tau)+\ln|\tau|^2+\ln(1+|z|^2)=\tilde{\Phi}(\tau z,\tau)+\ln(1+|\tau z|^2) \text{ for } (z,\tau)\in \mathbb P^1\times \overline{\mathbb D}^{\times}.$$
\end{proposition}
\begin{proof}
This is easily seen from the definition that $\Phi(z,\tau) + \ln |\tau|^2 + \ln (1+|z|^2) - \ln (1+|\tau z|^2)$ is a candidate for the envelope defining $\tilde{\Phi}(\tau z,\tau)$, giving one inequality and the other inequality is proved similarly.
\end{proof}

\subsection{Proof of Theorem \ref{thm:mainprecise}}

Without loss of generality we assume the arc $\gamma$ goes through the point $w=0$.   By Lemma \ref{lem:omega1}
$$\psi_1(z) = \ln \left(\frac{|z|^2}{1+|z|^2}\right)$$
and 
$$\Omega_1 = \mathbb P^1\setminus \gamma.$$
Looking at the other coordinate patch $\mathbb C_z$, we have that $\gamma$ is a curve passing through infinity, and so $\mathbb C_z\setminus\gamma$ is an open, simply connected proper subset of $\mathbb C_z$.  Hence by the Riemann-mapping theorem there is a biholomorphism
$$ f:\mathbb D\to \mathbb C_z\setminus \gamma \text{ with } f(0) = 0.$$
For $\tau \in \mathbb D^\times$ set
$$ A_{\tau} := f(\mathbb D_{|\tau|})\subset \mathbb C_z\subset \mathbb P^1.$$
Clearly each $A_{\tau}$ is a proper subset of $\mathbb C_{z}$ containing the origin, whose complement has non-empty interior.  

\begin{proposition}\label{prop:main}
We have
$$ \tilde{\Phi}(z,\tau)  = \psi_1(z) \text{ for all } \tau\in \mathbb D^{\times} \text{ and } z\in \mathbb P^1\setminus A_\tau.$$
\end{proposition}
\begin{proof}
By abuse of notation we write $\psi_1$ also for the pullback of $\psi_1$ to $\mathbb P^1\times \overline{\mathbb D}$.  Then
\begin{equation}
\tilde{\Phi}(z,\tau) \ge \psi_1(z) \text{ for } (z,\tau)\in \mathbb P^1\times \overline{\mathbb D}\label{eq:oneineq}
\end{equation}
since $\psi_1$ is a candidate for the envelope \eqref{eq:weaksolutionDtimes} defining $\tilde{\Phi}$.    

We next claim that\begin{equation} \label{eq:disceq}
\tilde{\Phi}(f(\tau),\tau)  = \psi_1(f(\tau)) \text{ for all } \tau \in \mathbb D.
\end{equation} 
To see this observe that $\tau\mapsto \tilde{\Phi}(f(\tau),\tau)$ is $f^*\omega_{FS}$-subharmonic and has Lelong number 1 at $\tau=0$.  On the other hand $\psi_1(f(\tau))$ is $f^*\omega_{FS}$-harmonic except at $\tau=0$ where it has Lelong number 1.  But $\tilde{\Phi}(f(\tau),\tau)$ tends to $\psi_1(f(\tau))$ as $|\tau|$ tends to $1$, and hence from the maximum principle along with \eqref{eq:oneineq} we get \eqref{eq:disceq}.

Now fix some $\tau \in \mathbb D^{\times}$ and set
$$\phi_\tau(z): = \tilde{\Phi}(z,\tau).$$
Then the above says that $\phi_{\tau} = \psi_1$ on $\partial A_{\tau}$.  On the other hand by \eqref{eq:oneineq} we have $\phi_{\tau} \ge \psi_1$ everywhere.  Moreover $\phi_{\tau}$ is $\omega_{FS}$-subharmonic on $A_{\tau}^c$ whereas $\psi_1$ is bounded and $\omega_{FS}$-harmonic on $A_{\tau}^c$.  Thus by the maximum principle we deduce $\phi_{\tau} = \psi_1$ on $A_{\tau}^c$ as required.





\end{proof}

\begin{proof}[Proof of Theorem \ref{thm:mainprecise}]
Set 
$$S := \{ (z,\tau)\in \mathbb P^1\times \mathbb D^{\times} : (\tau z,\tau)\in (A_{\tau}^c)^{\circ}\}$$
which is non-empty and open in $\mathbb P^1\times \mathbb D^{\times}$.  Then by Proposition \ref{prop:equiv} and then Proposition \ref{prop:main} if $(z,\tau) \in S$ we have

$$\Phi(z,\tau)= \tilde{\Phi}(\tau z, \tau) + \ln \left(\frac{1 + |\tau z|^2}{|\tau|^2(1+|z|^2)}\right) = 
\psi_1(\tau z) + \ln \left(\frac{1 + |\tau z|^2}{|\tau|^2(1+|z|^2)}\right).$$
Thus on $S$ we have
$$\pi_{\mathbb P^1} \omega_{FS} + dd^c \Phi = \pi_{\mathbb P^1} \omega_{FS} + dd^c\psi_1(\tau z) =0$$
as $\psi_1$ is $\omega_{FS}$-harmonic away from $z=0$.
\end{proof}

\subsection{A specific example}\label{sec:example}
We now construct  a specific potential $\phi$ that satisfies the hypothesis of Theorem \ref{thm:mainprecise}.   Fix $\gamma$ to be the interval $[-1,1]\subset \mathbb R\subset \mathbb C_w$.  Our goal is to find a $\phi\in C^{\infty}(\mathbb P^1)$ such that $\omega_{FS} + dd^c\phi>0$ and $\phi\ge - \ln (1+|w|^2)$ with equality precisely on $\gamma$.

To do so, let $\alpha:\mathbb R\to \mathbb R$ be a non-negative smooth non-decreasing convex function with $\alpha(t) =0$ for $t\le 1$ and $\alpha(t)>0$ for $t>1$.  Let
$$u(w): = \alpha(|w|^2)  + \operatorname{Im}(w)^2.$$
Thus $u$ is a smooth strictly subharmonic function on $\mathbb C_w$ that vanishes precisely on $\gamma$.  Then $\epsilon u-\ln (1+|w|^2)$  for some small constant $\epsilon>0$ is essentially the function that we want, we simply need to adjust it to have the correct behaviour far away from $\gamma$.  

 To do so we shall use a regularised version of the maximum function, which can be explicitly given as follows: Let $|\cdot|_{\operatorname{reg}}$ be a smooth convex function on $\mathbb R$ so that $|t|_{\operatorname{reg}} = |t|$ for $|t|\ge 1$.  Set $ \operatorname{max}_{\operatorname{reg}}(a,b) : = \frac{1}{2} ( |a-b|_{\operatorname{reg}} + a + b)$
and for $\delta>0$ put
\begin{equation} \operatorname{max}_\delta(a,b) := \delta \operatorname{max}_{\operatorname{reg}}(\delta^{-1} a, \delta^{-1}b).\label{eq:regmax}
\end{equation}
Then $\max_{\delta}(\cdot,\cdot)$ is smooth, and satisfies
$$ \operatorname{max}_{\delta}(a,b) = \left\{ \begin 
{array}{ll} a & \text{ if } a>b + \delta\\
 b & \text{ if } b> a+\delta.
 \end{array} \right.$$

Returning to the construction of $\phi$, fix sufficiently large constant $C$ and sufficiently small positive constant $\epsilon$ so that
\begin{align*}
 \epsilon u &\ge  \ln ( 1+ |w|^2) - C +1\text{ on } \mathbb D_{2}\\
 \epsilon u& \le \ln (1 + |w|^2) -C-1 \text{ on } \mathbb D_{4}\setminus \mathbb D_{3}.
 \end{align*}
Then for $0<\delta\ll 1$ set
 $$v: = \operatorname{max}_{\delta} (\epsilon u, \ln(1+|w|^2) -C).$$
 So $v$ is smooth, non-negative, strictly subharmonic,  equal to $\ln(1+|w|^2)-C$ on $\mathbb D_{4}\setminus \mathbb D_3$ and vanishes precisely on $\gamma$.   We then put
 $$\phi: = v - \ln(1+|w|^2)$$
 and extend $\phi$ to take the constant value $C$ in $\mathbb C_{w}\setminus \mathbb D_4$.  So $\phi$ extends to a smooth function over $\mathbb P^1$ with the desired properties.

\section{Discussion}\label{sec:discussion}


\subsection{What's going on?}  
Fix a $\phi\in C^{\infty}(\mathbb P^1)$ such that $\omega_{FS} + dd^c\phi>0$.  Then associated to $\phi$ we have two constructions:
\begin{enumerate}
\item The solution $\tilde{\Phi}$ to the complex HMAE on $\mathbb P^1\times \overline{\mathbb D}$ with boundary data given by $\phi_\tau = \phi$ for all $\tau\in \partial \mathbb D$ and the requirement of having Lelong number 1 at the point $(z,\tau) = (0,0)\in \mathbb C_z\times \mathbb D\subset \mathbb P^1\times \overline{\mathbb D}$.
\item The envelopes $\psi_t$ for $t\in (0,1]$ and the associated sets $\Omega_t(\phi)= \{ \psi_t<\phi\}$.
\end{enumerate}
In previous work of the authors we show that these sets of data are intimately connected.  First, $\tilde{\Phi}$ and $\psi_t$ are Legendre dual to each other \cite[Theorem 2.7]{RW} in that
\begin{equation} \label{legendre1}
\psi_t(z)=\inf_{|\tau|>0}\{\tilde{\Phi}(z,\tau)-(1-t)\ln|\tau|^2\}
\end{equation} 
and 
\begin{equation} \label{legendre2}
\tilde{\Phi}(z,\tau)=\sup_{t}\{\psi_{t}(z)+(1-t)\ln |\tau|^2\}.
\end{equation}
Second, the collection of sets $\Omega_t(\phi)$ that are biholomorphic to a disc describe the harmonic discs of $\tilde{\Phi}$.  That is, if $t$ is such that $\Omega_t(\phi)$ is a proper simply connected subset of $\mathbb C_z$ and $f:\mathbb D\to \Omega_t$ is a Riemann-map with $f(0) =0$ then the restriction of $\tilde{\Phi}$ to the graph $\{ (f(\tau),\tau)\in \mathbb P^1\times \mathbb D\}$ is $\omega_{FS}$-harmonic.  Furthermore it is shown in \cite[Theorem 3.1]{RW} these are (essentially) the only harmonic discs that occur.

We can say more.  For $\tau\in \mathbb D^{\times}$ set  
$$\phi_{\tau}(z): = \tilde{\Phi}(z,\tau).$$    If $\tilde{\Phi}$ is regular then each $\phi_{\tau}$ will be smooth K\"ahler potential, but in general this will not be the case. Nevertheless, by the work of B\l ocki \cite{Blocki}  we know $\phi_{\tau}$ is $C^{1,1}$ and since $\pi_{\mathbb P^1}^*\omega_{FS} + dd^c\tilde{\Phi}\ge 0$ we know $\omega_{FS} + dd^c\phi_{\tau}$ is semipositive.  We can then define the associated sets $\Omega_t(\phi_{\tau})$ in exactly the same way as before.

\begin{proposition}\label{prop:innerdisc}
Suppose $t$ is such that $\Omega_t(\phi)\subset \mathbb C_z$ is proper and simply connected and let $f_t:\mathbb D\to \Omega_t(\phi)$ be a Riemann-map with $f(0)=0$.  Then for each $\tau\in \mathbb D^\times$ we have
$$ f_t(\mathbb D_{|\tau|}) = \Omega_t(\phi_{\tau}).$$
\end{proposition}
 We shall give a proof of this fact below, but assuming it for now we can give an alternative proof that, under the hypothesis of Theorem \ref{thm:mainprecise}, for each $\tau\in \mathbb D^\times$ the current $\omega_{FS} + dd^c\tilde{\Phi}(\cdot,\tau)$ is degenerate on some non-empty open subset of $\mathbb P^1$.   First Lemma \ref{lem:omega1} gives 
 $$\Omega_1(\phi) =\mathbb P^1\setminus \gamma$$
which is a simply connected proper subset of $\mathbb C_z.$
We then take our Riemann-map $f:\mathbb D\to \Omega_1(\phi)$ and consider the image
$$ A_\tau: = f(\mathbb D_{|\tau|}) = \Omega_1(\phi_\tau)\text{ for } \tau\in \mathbb D^{\times}.$$
As observed before, $A_{\tau}$ is a proper subset of $\mathbb C_{z}$ whose complement has non-empty interior.  

On the other hand, it is a general fact that for each $t$ the set  $\Omega_t(\phi_\tau)$ has measure $t$ with respect to the current $\omega_{FS} + dd^c\phi_\tau$.  (If $\phi_\tau$ is smooth and $\omega_{FS} + dd^c\phi_\tau$ is strictly positive this is a standard piece of potential theory and discussed in  \cite[Proposition 1.1]{RW}; when $\phi_\tau$ is merely $C^2$ and $t<1$ then this is proved in \cite[Theorem 1.2]{RWEnvelopes} and the case $t=1$ follows from this by continuity as $\Omega_1(\phi_\tau)=\cup_{t<1}\Omega_t(\phi_\tau)$; finally when $\phi_\tau$ is merely $C^{1,1}$ this is given in \cite[Remark 1.19, Corollary 2.5]{BermanDemailly}.)  

Therefore
$$\int_{A_\tau} (\omega_{FS} + dd^c\phi_\tau) = \int_{\Omega_1(\phi_\tau)} (\omega_{FS} + dd^c\phi_\tau) = 1.$$
But our normalisation is that $\int_{\mathbb P^1} (\omega_{FS} + dd^c\phi_{\tau}) = \int_{\mathbb P^1} \omega_{FS} =1$ as well, and so the current $\omega_{FS} + dd^c\phi_{\tau}$ gives zero measure to the complement of $A_{\tau}$, which is precisely what we were aiming to prove.

\begin{proof}[Proof of Proposition \ref{prop:innerdisc}]
Fix $\sigma\in \mathbb D^{\times}$ and set $r: = |\sigma|$.  Our aim is to show
$$ f_t(\mathbb D_r) = \Omega_t(\phi_\sigma).$$

 Consider the $S^1$-action on $\mathbb P^1\times \overline{\mathbb D}$ given by $e^{i\theta}\cdot (z,\tau) = (z,e^{i\theta}\tau)$, and observe that the boundary data used to define $\tilde{\Phi}$ from \eqref{eq:weaksolutionDtimes} is $S^1$-invariant, which implies $\tilde{\Phi}$ is $S^1$-invariant as well.   Thus we may as well assume that $\sigma$ is real, so $\phi_{\sigma} = \phi_r$.  

For a function $F$ on $\mathbb P^1\times \overline{\mathbb D}$ and $D\subset \overline{\mathbb D}$ we write $F|_{D}$ for the restriction of $F$ to $\mathbb P^1\times D$. Then $\tilde{\Phi}|_{\overline{\mathbb D}_r}$ is the solution to the Dirichlet problem for the HMAE with boundary data $(\phi_\tau)_{\tau\in \partial \mathbb D_r} = \phi_{r}$ and the requirement that $\tilde{\Phi}|_{\overline{\mathbb D}_r}$ has Lelong number 1 at the point $(0,0)\in \mathbb C_z\times \mathbb D_r\subset \mathbb P^1\times \mathbb D_r$.

Letting $s:=-\ln |\tau|^2$ consider the function on $\mathbb P^1\times \overline{\mathbb D}^\times$ given by
 $$H(z,\tau):=\frac{\partial}{\partial s} \tilde{\Phi}(z,e^{-s/2}),$$ (when $|\tau|=1$ and thus $s=0$ we take the right derivative). As $\tilde{\Phi}$ is $C^{1,1}$ on $\mathbb{P}^1\times \overline{\mathbb{D}}^{\times}$ the function $H$ is well-defined and Lipschitz.    Clearly this is compatible with restriction, i.e.\
 $$ H|_{\overline{\mathbb D}^\times_r}(z,\tau) = \frac{\partial}{\partial s} \tilde{\Phi}|_{\overline{\mathbb D}_r}(z,e^{-s/2}).$$

Now, as discussed above, and proved in  \cite[Theorem 3.1]{RW}, the graph $\{ (f(\tau),\tau) : \tau \in \mathbb D\}$ of $f$ is a harmonic disc for $\tilde{\Phi}$.  What is also proved is that $H$ takes the constant value $t-1$ along this disc so
$$ H(f(\tau),\tau) = t-1 \text{ for }  \tau \in \mathbb D^\times.$$
Now $H$ is also $S^1$-invariant and so this in particular implies
$$ H(f(re^{i\theta}),r)  = H(f(re^{i\theta}), re^{i\theta}) = t-1 \text{ for all } \theta \in [0,2\pi].$$
In other words the function $H(\cdot, r)$
takes the value $t-1$ on the boundary of $f(\mathbb D_r)$.    On the other hand, we prove in  \cite[Proposition 2.9]{RW} that the function $H(\cdot, r)$ describes the set $\Omega_t(\phi_r)$, in that
$$ H(z,r) + 1 = \sup\{ s: z\notin \Omega_s(\phi_r)\}$$
(we remark that the proof of \cite[Proposition 2.9]{RW} does not require any regularity or  strict positivity assumptions on the potential $\phi_\sigma$).  Thus $\Omega_t(\phi_r)$ is the interior component 
of the curve $\theta\mapsto f(re^{i\theta})$ (that is, the component containing $z=0$), which gives $\Omega_t(\phi_r) = f(\mathbb D_r)$ as desired.
\end{proof}

\subsection{Extensions and Questions}

Under they hypothesis of Theorem \ref{thm:mainprecise} we have shown that the current $\omega_{FS}+ dd^c\Phi(\cdot,\tau)$ fails to be strictly positive on any interior fibre (that is for any $\tau$ with $0<|\tau|<1$).  Furthermore we have no reason to expect our solution to be smooth everywhere.  Thus the following two questions are natural:

\begin{question}
Does there exist a smooth family of potential $(\phi_\tau)_{\tau\in \partial B}$ for which the solution to the complex HMAE \eqref{eq:hmae} is everywhere smooth but not of maximal rank?
\end{question}

\begin{question}
Does there exist a smooth family of potential $(\phi_\tau)_{\tau\in \partial B}$ for which the solution to the complex HMAE \eqref{eq:hmae} such that $\omega + dd^c\Phi(\cdot,\tau)$ is a K\"ahler form for some $\tau$ with $0<|\tau|<1$ but not for others.
\end{question}

 
We are not currently able to answer these questions. However we believe that the degenerate solutions we describe in this paper are actually regular in the interior of the complement of the degenerate set $S$ (that is, they are smooth there and of maximal rank).    In fact from our previous work in \cite{RW} we can understand the set on which our solution is regular in terms of the collection of sets $\Omega_t(\phi)$ that are simply connected.  Now, our specific potential $\phi$ (Section \ref{sec:example}) was constructed to have curvature equal to $\omega_{FS}$ far away from the arc $\gamma = [-1,1]\subset \mathbb C_w\subset \mathbb P^1$, from which one can see that $\Omega_t(\phi)$ is a disc for sufficiently small $t$.  This gives an open set of $\mathbb P^1\times \mathbb D$ for which the solution $\Phi$ is regular.  Furthermore, by construction, $\Omega_1(\phi)$ is simply connected.  We think it likely that $\Omega_t(\phi)$ is actually simply connected for all $t$, which would give rather precise information about the set on which our solution is regular, but does not seem easy to prove that this is the case. 
 
 We furthermore believe that the fibrewise Laplacian of such a solution is uniformly bounded from below on the complement of $S$, so has a discontinuity on the boundary $\partial S$ where it jumps to zero.   A somewhat bold conjecture would be that any solution to the HMAE is regular away from the set where it fails to have maximal rank, and is smooth away from the boundary of this set.

\medskip
\small{
\noindent {\sc Julius Ross,  DPMMS , University of Cambridge, UK\\ j.ross@dpmms.cam.ac.uk}\medskip

\noindent{\sc David Witt Nystr\"om, 
Department of Mathematical Sciences, Chalmers University of Technology and the University of Gothenburg, Sweden \\ wittnyst@chalmers.se, danspolitik@gmail.com}

}

\end{document}